\documentclass{amsart}
\usepackage{amssymb}
\usepackage{amsmath}
\usepackage[centertags]{amsmath}
\usepackage[english,francais]{babel}

\newtheorem{theorem}{Theorem}[section]
\newtheorem{lemma}[theorem]{Lemma}
\newtheorem{proposition}{Proposition}
\newtheorem{corollary}{Corollary}
\theoremstyle{definition}

\newtheorem{example}{Example}
\newtheorem{remark}{Remark}

\numberwithin{equation}{section}

\begin{document}

\title[Stability under deformations of Hermite-Einstein metrics]{ Stability under deformations of Hermite-Einstein almost K\"ahler metrics}

\author{{Mehdi Lejmi}}
\thanks{The work of the author was supported by FQRNT grant.}
\address{ School of Mathematics\\
University of Minnesota \\ Minneapolis (Minnesota) \\
55455 \\ USA} \email{mlejmi@umn.edu}

\maketitle

\medskip

\selectlanguage{english}
\begin{abstract}
On a $4$-dimensional compact symplectic manifold, we consider a smooth family of compatible almost-complex structures such that at time zero the induced metric is Hermite-Einstein almost-K\"ahler metric with zero or negative Hermitian scalar curvature.
We prove, under certain hypothesis, the existence of a smooth family of compatible almost-complex structures, diffeomorphic at each time to the initial one, and inducing constant Hermitian scalar curvature metrics. 
\end{abstract}
\section{Introduction}
On a $2n$-dimensional symplectic manifold $(M,\omega)$, an almost-complex structure $J$ is $\omega$-compatible if it induces a Riemannian metric $g$ via the relation $g(\cdot,\cdot)=\omega(\cdot,J\cdot)$. The metric $g$ is called then an almost-K\"ahler metric. When $J$ is integrable, the induced metric is K\"ahler.
Given an $\omega$-compatible almost-complex structure $J$, there exists a canonical Hermitian connection with torsion $\nabla$ \cite{gau1, lib} on the tangent bundle, which preserves both $\omega$ and $J$. The curvature of the induced Hermitian connection on the anti-canonical bundle is of the form $\sqrt{-1}\rho^\nabla$, where $\rho^\nabla$ is a closed real $2$-form called the Hermitian Ricci form. The Hermitian scalar
curvature $s^\nabla$ is defined as the contraction of $\rho^\nabla$ by $\omega$ and coincides with the (usual) Riemannian scalar curvature when the metric is K\"ahler. An almost-K\"ahler metric is
called Hermite-Einstein \cite{lej} (HEAK for short) if the Hermitian Ricci form $\rho^{\nabla}$ satisfies $\rho^{\nabla}=\frac{s^\nabla}{2n}\omega$ (in particular $s^\nabla$ is a constant). Note that the terminology `Hermite-Einstein' here {\it does not imply the integrability} of the almost-complex structure.

On a compact symplectic manifold $(M,\omega)$, we consider the space $AK_\omega$ of all $\omega$-compatible almost-complex structures. This is an infinite dimensional Fr\'echet space equipped with a formal K\"ahler structure described by Fujiki \cite{don,fuj}. Furthermore, there is a natural action of the Hamiltonian group $Ham(M,\omega)$ on $AK_\omega$ and it turns out that this action is Hamiltonian \cite{don,fuj}
with moment map identified with the Hermitian scalar curvature. A metric induced by a critical point of the square norm of the moment map $J\mapsto \int_M(s^\nabla)^2{\omega^n}$ is called an extremal almost-K\"ahler metric \cite{apo-dra,lej,lej1}. Moreover, an almost-K\"ahler metric induced by $J$ is extremal if and only if the symplectic gradient of its Hermitian scalar curvature is an infinitesimal isometry of $J$. Extremal almost-K\"ahler metrics are a generalization of Calabi extremal K\"ahler metrics \cite{cal}.
Furthermore, almost-K\"ahler metrics with constant Hermitian scalar curvature are extremal. 

In the K\"ahler setting, Fujiki--Schumacher \cite{fuj-sch} and Lebrun--Simanca \cite{leb-sim} showed, in the abscence of holomorphic vector fields, that the existence of extremal K\"ahler metrics is an open condition. Moreover, Apostolov--Calderbank--Gauduchon--Friedman \cite{apo-cal-gau-fri 2} proved the openess by fixing a maximal torus in the reduced automorphism group of the complex manifold $(M,J)$. Furthermore, Rollin--Simanca--Tipler \cite{rol-sim-tip}
showed with a certain hypothesis the stability of extremal K\"ahler metrics under complex deformations and hence generalized the results of \cite{leb-sim,leb-sim-1} (see also \cite{rol-tip,sze}).

In the general almost-K\"ahler case, one expects, from the GIT standard picture \cite{don,fog-kir-mum},
the existence and uniqueness of extremal almost-K\"ahler metrics, up to the action of $Ham(M,\omega)$, in every `stable complexified' orbit of the action of $Ham(M,\omega)$. The complexification of $Ham(M,\omega)$ does not exist. However, one can complexify the action on the level of the Lie algebra and once we are restricted to the integrable $\omega$-compatible almost-complex structures, a description of this complexified orbit is given when $H^1(M,\mathbb{R})=0$ \cite{don}. It is identified with the space of K\"ahler forms in the cohomology class of $\omega$.

In a previous paper \cite{lej1}, on a compact symplectic $4$-manifold $(M,\omega)$, we considered a smooth path of $\omega$-compatible almost-complex structures $J_t $ invariant under a (fixed) maximal torus $T$ in $Ham(M,\omega)$ such that $J_0$ induces an extremal K\"ahler metric. In particular, $J_0$ {\it is integrable}. Furthermore, we supposed that $h^-_{J_t}=b^+-1$ for sufficiently small $t$, where $h^-_{J_t}$ is the dimension of $g_t$-harmonic $J_t$-anti-invariant $2$-forms \cite{dra-li-zha} (here $g_t$ is the metric induced by $J_t$). Then, we showed, for a short time, the existence of smooth
family of $T$-invariant $\omega$-compatible almost-complex structures $\tilde{J_t}$ inducing extremal almost-K\"ahler metrics such that $\tilde{J}_0=J_0$ and $\tilde{J}_t$ is diffeomorphic to $J_t$ for each $t$. In the spirit of Lebrun--Simanca result \cite{leb-sim}, the proof consists mainly to deform the symplectic form by introducing a notion of almost-K\"ahler potential (defined only in dimension $4$) and then using the Banach implicit function theorem for the Hermitian scalar curvature map. The hypothesis $h^-_{J_t}=b^+-1$ was necessarily to insure the continuity of the Hermitian scalar curvature map since a family of Green operators is involved in the definition of this almost-K\"ahler potential. A recent result of Tan--Wang--Zhang--Zhu \cite{tan-wan-zha-zhu}
implies that one can drop the assumption $h^-_{J_t}=b^+-1$.

Now, if we suppose that $J_0$ {\it is not integrable}, it is not clear how to identify the kernel of the linearization of the Hermitian scalar curvature map with the Lie algebra of Hamiltonian Killing vector fields even in the simplest case namely when $J_0$ induces a HEAK metric.
The idea in this paper is to define another suitable almost-K\"ahler potential for which it is possible to study the kernel of the derivative of the Hermitian scalar curvature map at least in the latter case.
The almost-K\"ahler potential defined in this paper follows from a generalization of the $dd^c$-Lemma \cite{del-gri-mor-sul} to the almost-K\"ahler case. For instance, one can derive a Hodge decomposition of the Riemannian dual of a (real) holomorphic vector field
on a compact almost-K\"ahler manifold. When $J_0$ induces a HEAK metric with zero or negative Hermitian scalar curvature, we obtain the following


%


\begin{theorem}\label{th2}
Let $(M,\omega)$ be a $4$-dimensional  compact symplectic manifold.
Let $J_t$ be any smooth family of $\omega$-compatible almost-complex structures such that $J_0$ induces a HEAK metric with zero or negative Hermitian scalar curvature. Moreover, suppose that  for a small $t$, $h^-_{J_t}=h^-_{J_0}=b^+-1.$
Then, there exists a smooth family of $\omega$-compatible almost-complex structures  ${\tilde{J}}_t$, defined for small $t$, inducing almost-K\"ahler metrics with constant Hermitian scalar curvature such that
$\tilde{J}_t$ is diffeomorphic to $J_t$ for each $t$ and $\tilde{J}_0=J_0$.
\end{theorem}

We note that, in the above theorem, the condition $h^-_{J_t}=h^-_{J_0}=b^+-1$ is not to ensure the continuity of the Hermitian scalar curvature map but to guarantee the $J_t$-invariance of the constructed symplectic forms. Moreover, by \cite{tan-wan-zha-zhu}, one has only to suppose that $h^-_{J_0}=b^+-1$.
The latter condition is satisfied in the cases mentionned in \cite{lej1}. Moreover, by a result of Li and Tomassini \cite{li-tom}, any homogeneous almost-K\"ahler structure $(\omega,J)$ on a $4$-dimensional compact manifold $M=G/\Gamma$,
where $G$ is a simply-connected Lie group and $\Gamma\subset G$ a uniform discrete subgroup, verifies $h^-_{J}=b^+-1.$ Namely, the Kodaira--Thurston manifold has non-integrable almost-K\"ahler
metrics $(\omega,J)$ with vanishing Hermitian Ricci form satisfying the condition $h^-_{J}=b^+-1$ \cite{dra,li, vez}.

%
%

%

%

\section{Preliminaries}
Let $(M,\omega)$ be a symplectic manifold of dimension $2n$. An almost-complex structure $J$ is $\omega$-compatible if the induced $2$-tensor field $g(\cdot,\cdot):=\omega(\cdot,J\cdot)$ is a Riemannian metric. Then, $g$ is called an ($\omega$-compatible) almost-K\"ahler metric. In the rest of the paper, we identify the induced metric $g$ with the couple $(\omega,J)$. If, additionally, $J$ is {{integrable}}, then $(\omega,J)$ is a {K\"ahler} metric.

The almost-complex structure $J$ acts on the cotangent bundle $T^\ast M$ by $J\alpha(X)=-\alpha(JX),$ where $\alpha$ is a $1$-form and $X$ a vector field on $M$. The action of $J$ can be extend to any $p$-form $\psi$
by $(J\psi)(X_1,\cdots,X_p)=(-1)^p\psi(JX_1,\cdots,JX_p)$. 
The bundle of $2$-forms $\Lambda^2M$ decomposes under the action of $J$ as follows
\begin{equation}\label{eq13}
\Lambda^2M=\mathbb{R}\,.\,\omega\oplus\Lambda^{J,-}M\oplus\Lambda_0^{J,+}M,
\end{equation}
where $\Lambda^{J,-}M$ is the subbundle of $J$-anti-invariant $2$-forms and $\Lambda_0^{J,+}M$ is the subbundle of $J$-invariant $2$-forms pointwise orthogonal to $\omega.$ 

For an almost-K\"ahler metric $(\omega,J)$, the Hermitian connection $\nabla$ on $(TM,\omega,J)$ is defined by
\begin{equation*}
\nabla_XY=D^g_XY-\frac{1}{2}J\left(D^g_XJ\right)Y,
\end{equation*}
where $D^g$ is the Levi-Civita connection with respect to the induced metric $g$ and $X,Y$ are vector fields on $M$. Let $R^{\nabla}$ be the curvature of $\nabla$. Then, the {\it Hermitian Ricci form} $\rho^{\nabla}$ is defined by
\begin{equation*}
\rho^{\nabla}(X,Y)=-tr(J\circ R^{\nabla}_{X,Y}),
\end{equation*}
where $R^{\nabla}_{X,Y}$ is viewed as an anti-Hermitian linear operator of $(TM,\omega,J)$. The form $\rho^{\nabla}$ is a deRham representative of $2\pi c_1(TM,J)$ in $H^2(M,\mathbb{R})$, where $c_1(TM,J)$ is the first (real) Chern class. 
If we suppose that $\omega$ and $\tilde{\omega}$ are symplectic forms compatible with the same almost-complex structure $J$ and satisfy $\tilde{\omega}=e^F\omega^n$ for some real-valued function $F$ then
\begin{equation}\label{eq23}
\tilde{\rho}^\nabla=-\frac{1}{2}{dJdF}+\rho^\nabla,
\end{equation}
where $\tilde{\rho}^\nabla$ (resp. ${\rho}^\nabla$) is the Hermitian Ricci form of $(\tilde{\omega},J)$ (resp. $({\omega},J)$).

We define the {\it Hermitian scalar curvature} $s^{\nabla}$ of an almost-K\"{a}hler metric $(\omega,J)$ as the trace of $\rho^\nabla$ with respect to $\omega$, i.e.
\begin{equation}\label{eq19}
s^{\nabla}\omega^n=2n\left(\rho^{\nabla}\wedge\omega^{n-1}\right).
\end{equation}

An almost K\"ahler metric $(\omega,J)$ is called {\it Hermite-Einstein} (HEAK for short) if $$\rho^{\nabla}=\frac{s^\nabla}{2n}\omega.$$ In particular, $s^\nabla$ is a constant.


The Riemannian Hodge operator $\ast_g:\Lambda^pM\rightarrow\Lambda^{2n-p}M$ is defined to be the unique isomorphism such that $\psi_1\wedge(\ast_g)\psi_2=g(\psi_1,\psi_2)\frac{\omega^n}{n!},$ for any $p$-forms $\psi_1$ and $\psi_2$.
Moreover, since the dimension of $M$ is even, $(\ast_g)^2\psi=(-1)^p\psi$ on $p$-form $\psi$.
In dimension $4$, the bundle of $2$-forms decomposes as
\begin{equation*}
\Lambda^2M=\Lambda^+M\oplus\Lambda^+M,
\end{equation*}
where $\Lambda^\pm M$ corresponds to the eignevalue $(\pm 1)$ under the action of the Riemannian Hodge operator $\ast_g$. This decomposition is related to the splitting \ref{eq13} in the following way
\begin{equation*}
\Lambda^+M=\mathbb{R}\,\omega\oplus\Lambda^{J,-}M \,\,\,\,\mbox{   and   }\,\,\,\Lambda^-M=\Lambda_0^{J,+}M.
\end{equation*} 

\section{Generalized ${dd^c}$-Lemma}

In this section, we generalize the $dd^c$-Lemma \cite{del-gri-mor-sul} to the almost-K\"ahler case. For this purpose, we need to present some symplectic commutators.

Let $(M,\omega,J,g)$ be a compact almost-K\"ahler manifold of dimension $2n$. Let $\delta^g$ be the {\it codifferential} defined as the formal adjoint of the Levi-Civita connection $D^g$ with respect to the almost-K\"ahler metric $g$ when it is applied to sections of $\otimes^pT^\ast M$. In particular, it is the
adjoint of the exterior derivative $d$ when it is applied to $p$-forms and are related by $\delta^g=-\ast_gd\,\,\ast_g$ since the dimension of $M$ is even. Denote by $\Delta^g=d\delta^g+\delta^gd$ the Laplacian and $\mathbb{G}$ the Green operator associated to $\Delta^g$. The Riemannian Hodge operator $\ast_g$ commutes with $\Delta^g$. It follows that $\ast_g$ commutes with $\mathbb{G}$. 

Let $\delta^c=(-1)^pJ\delta^gJ$ be the {\it twisted codifferential} acting on $p$-forms. This is the {\it symplectic adjoint} of $d$. Define the {\it twisted differential} $d^c$ by $d^c=(-1)^pJdJ$ acting on $p$-forms and let $\Delta^c=d^c\delta^c+\delta^cd^c$ be the {\it twisted Laplacian} and $\mathbb{G}^c$ the Green operator associated to $\Delta^c.$ One can prove in elementary way that the codifferential $\delta^g$ and the exterior derivative $d$ (resp. $\delta^c$ and $d^c$) commute with $\mathbb{G}$ (resp. $\mathbb{G}^c$).

%

We denote by $\Lambda_\omega$ the contraction by the symplectic form $\omega$ defined for a $p$-form $\psi$ by
$\Lambda_\omega(\psi)=\frac{1}{2}\sum_{i=1}^{2n}\psi(e_i,Je_i,\cdots)$, where $\{e_i\}$ is a local $J$-adapted orthonormal frame. The commutator of $\Lambda_\omega$ and $d^c$ is given by \cite{gau,mer}
\begin{equation}\label{eq20}
[\Lambda_\omega,d^c]=\delta^g.
\end{equation}
It follows that $d^c\delta^g+\delta^gd^c=0.$ Furthermore, since $\Lambda_\omega$ commutes with $J$, the relation \ref{eq20} implies \cite{gau}
\begin{equation}\label{eq21}
[\Lambda_\omega,d]=-\delta^c.
\end{equation}

Moreover, if $L_\omega$ is the adjoint of $\Lambda_\omega$ acting on a $p$-form $\psi$ by $L_\omega\psi=\omega\wedge\psi$, then \cite{gau} 
\begin{equation}\label{eq12}
[L_\omega,\delta^g]=d^c.
\end{equation}

%
Now, we are in position to derive a generalization of the the $dd^c$-Lemma. 
\begin{lemma}\label{l1}
On a compact almost-K\"ahler manifold, let $\psi$ be any $J$-invariant $p$-form satisfying $\psi=d\phi$ for some $(p-1)$ form $\phi$. Then,
\begin{eqnarray*}
\psi=d\mathbb{G}d^c\tilde{\psi}=\mathbb{G}dd^c\tilde{\psi},
\end{eqnarray*}
for some $(p-2)$-form $\tilde{\psi}$.
\end{lemma}
\begin{proof}
It follows from the Hodge decomposition with respect to $\Delta^g$ of $\psi$ and since $d\psi=0$ that 
\begin{equation}\label{eq41}
\psi=(\psi)_H+d\delta^g\mathbb{G}\psi=d\delta^g\mathbb{G}\psi.
\end{equation}
Recall that $(\psi)_H$=0 because $\psi$ is $d$-exact.

On the other hand, $d^c\psi=0$ since $\psi$ is $J$-invariant. So, it follows from the Hodge decomposition with respect to $\Delta^c$ that $\psi=(\psi)_{H^c}+d^c\delta^c\mathbb{G}^c\psi$ (here $(\psi)_{H^c}$ denotes the harmonic part with respect to $\Delta^c$, in particular $d^c(\psi)_{H^c}=\delta^c(\psi)_{H^c}=0$). Plugging this in \ref{eq41}, we obtain 
\begin{eqnarray*}
\psi&=&d\delta^g\mathbb{G}\Big{(}\big{(}(\psi)_{H^c}\big{)}+d^c\delta^c\mathbb{G}^c\psi\Big{)}, \\
&=&d\delta^g\mathbb{G}\big{(}(\psi)_{H^c}\big{)}+d\delta^g\mathbb{G}(d^c\delta^c\mathbb{G}^c\psi), \\
&=&d\mathbb{G}\delta^g\big{(}(\psi)_{H^c}\big{)}+d\mathbb{G}\delta^g(d^c\delta^c\mathbb{G}^c\psi),\\
&=&d\mathbb{G}[\Lambda_\omega,d^c]\big{(}(\psi)_{H^c}\big{)}-d\mathbb{G}d^c(\delta^g\delta^c\mathbb{G}^c\psi),\\
&=&-d\mathbb{G}d^c\Lambda_\omega((\psi)_{H^c})-d\mathbb{G}d^c(\delta^g\delta^c\mathbb{G}^c\psi),\\
&=&d\mathbb{G}d^c\Big{(}-\Lambda_\omega\big{(}(\psi)_{H^c}\big{)}-\delta^g\delta^c\mathbb{G}^c\psi\Big{)}.
\end{eqnarray*}

Here, we used the equality \ref{eq20} and the fact that $d^c\delta^g+\delta^gd^c=0$. The Lemma follows.
\end{proof}
In the K\"ahler case, remark that  $\Delta=\Delta^c$ so
$\mathbb{G}=\mathbb{G}^c$. Hence, $d^c\mathbb{G}=\mathbb{G}d^c$. Then, $\psi=d\mathbb{G}d^c\tilde{\psi}=dd^c(\mathbb{G}\tilde{\psi}).$

\begin{proposition}\label{p2}
On a compact almost-K\"ahler manifold, let $\psi_1,\psi_2$ be any two real $J$-invariant closed $2$-forms and suppose that $\psi_1,\psi_2$ determine the same deRham cohomology class. Then, there exists a real function $f$, uniquely defined up to an additive constant, such that
\begin{equation*}
\psi_1-\psi_2=d\mathbb{G}d^cf=\mathbb{G}dd^cf.
\end{equation*}
\end{proposition}
\begin{proof}
This is a direct application of Lemma \ref{l1} for $\psi=\psi_1-\psi_2$. If $\mathbb{G}dd^cf=0$ then $dd^cf$ is harmonic. Since $M$ is compact, $dd^cf=0$. By the equality \ref{eq20}, it follows that $\Delta^gf=\delta^gdf=[\Lambda_\omega, d^c]df=\Lambda_\omega d^cdf=0$ (because $d^cdf=-Jdd^cf$). So $f$ is constant as $M$ is compact.
\end{proof}

As a consequence of Lemma \ref{l1}, we obtain a Hodge decomposition of the Riemannian dual of a (real) holomorphic vector field on a compact almost-K\"ahler manifold $(M,\omega,J,g)$. Recall that a (real) vector field $X$ is called holomorphic
if it is an infinitesimal isometry of $J$ i.e. $\mathfrak{L}_XJ=0,$ where $\mathfrak{L}$ is the Lie derivative.

\begin{corollary}\label{c2}
Let $X$ be a holomorphic vector field on a compact almost-K\"ahler manifold and $\xi=X^{\flat_g}$ the dual of $X$ with respect to the metric $g.$ Then, we have
\begin{eqnarray}
\xi=(\xi)_{H^c}+d^cu-J\mathbb{G}d^cv,
\end{eqnarray}
where $u,v$ are real functions, uniquely defined up to an additive constant. Here $(\xi)_{H^c}$ denotes the harmonic part with respect to $\Delta^c.$
\end{corollary}
Remark that in the K\"ahler case, $(\xi)_{H^c}=\xi_H$ and $-J\mathbb{G}d^cv=d(\mathbb{G}v).$ 
\begin{proof}
Since $X$ is holomorphic, a direct computation shows that $\mathfrak{L}_X\omega$ is a $J$-invariant $2$-form.
Hence, $dJ\xi$ is $J$-invariant. By Lemma \ref{l1}, $dJ\xi=d\mathbb{G}d^cv,$ for a function $v$ uniquely defined up to a constant.
The Hodge decomposition with respect to $\Delta^c$ of $\xi$ is given by $\xi=(\xi)_{H^c}+d^cu+\delta^c\phi$ for some real function $u$ and $2$-form $\phi$. Then 
\begin{equation}\label{eq14}
dJ\xi=-d\delta^gJ\phi=d\mathbb{G}d^cv. 
\end{equation}
Moreover, using equality \ref{eq12}, we have $d\mathbb{G}d^cv=-d\mathbb{G}\delta^g(v\omega)=-d\delta^g\mathbb{G}(v\omega)$. So, from \ref{eq14} we have $-d\delta^gJ\phi=-d\delta^g\mathbb{G}(v\omega)$, thus $-\delta^gJ\phi=-\delta^g\mathbb{G}(v\omega)=\mathbb{G}d^cv$. The Corollary follows.
\end{proof}

Now, given any function $f$, it is natural to wonder whether $d\mathbb{G}d^cf$ is $J$-invariant.
\begin{proposition}\label{p3}
In dimension $2n=4$, for any smooth function $f$, 
\begin{equation*}
(d\mathbb{G}d^cf)^{J,-}=\frac{1}{2} (f_0\omega)_H-\frac{1}{4}g\left((f_0\omega)_H,\omega\right)\,\omega.
\end{equation*}
In particular, if $h^-_J=b^+-1$, then $d\mathbb{G}d^cf$ is $J$-invariant (here $f_0$ is the orthogonal projection of $f$ onto the complement of the constants).
\end{proposition}
\begin{proof}
Using the equality \ref{eq12} and the fact that the Hodge operator $\ast_g$ commutes with $\mathbb{G}$, we have
\begin{eqnarray*}
(d\mathbb{G}d^cf)^{J,-}&=&(-d\mathbb{G}\delta^g(f\omega))^{J,-}\\
&=&\frac{1}{2}(I+\ast_g)(-\mathbb{G}d\delta^g(f\omega))-\frac{1}{4}g\left((I+\ast_g)(-\mathbb{G}d\delta^g(f\omega)),\omega\right)\,\omega\\
&=&-\frac{1}{2}\mathbb{G}\Delta^g(f\omega)+\frac{1}{4}g\left(\mathbb{G}\Delta^g(f\omega),\omega\right)\,\omega\\
&=&-\frac{1}{2} f\omega+\frac{1}{2} (f\omega)_H+\frac{1}{2} f\omega-\frac{1}{4}g\left((f\omega)_H,\omega\right)\,\omega\\
&=&\frac{1}{2} (f\omega)_H-\frac{1}{4}g\left((f\omega)_H,\omega\right)\,\omega\\
&=&\frac{1}{2} (f_0\omega)_H-\frac{1}{4}g\left((f_0\omega)_H,\omega\right)\,\omega
\end{eqnarray*}
Here, we use the convention $g(\omega,\omega)=2$. In case when $h^-_J=b^+-1$, we have $(f_0\omega)_H=0$.
Indeed, under the latter assumption, for any $g$-harmonic $2$-form $\psi$ (with respect to $\Delta^g$), the pairing $g(\omega,\psi)$ is a constant function. Thus, given a function $f$, we obtain $<(f_0\omega)_H,\psi>_{L_2}=\int_Mf_0g(\omega,\psi)\frac{\omega^2}{2!}=0.$ Hence, $(f_0\omega)_H=0$ and therefore $d\mathbb{G}d^cf$ is $J$-invariant. 

\end{proof}

Thus, in dimension $4$, when $h_J^-=b^+-1$, the symplectic form $\omega+d\mathbb{G}d^cf$ is $J$-invariant for any function $f$ and so $f$ is called {\it almost-K\"ahler potential} when $(\omega+d\mathbb{G}d^cf,J)$ induces a Riemannian metric. Again remark that in the K\"ahler case, $d\mathbb{G}d^cf=dd^c\mathbb{G}f$, hence $\mathbb{G}f$ coincides with the usual K\"ahler potential.

The `potential' $\mathbb{G}f$ coincides with the potential defined by Weinkove \cite{wei} in the following way: Given a symplectic form $\tilde{\omega}$ compatible with $J$ and cohomologous to $\omega$, then $\tilde{\omega}-\omega=d\mathbb{G}d^cf$, for some function $f$. Now, let $\phi=\mathbb{G}f$ then a direct computation shows that $$d\mathbb{G}d^cf=d\mathbb{G}d^c\Delta^g\phi=dd^c\phi- 2d\delta^g\mathbb{G}D_{(d\phi)^{\sharp_g}}^g\omega.$$ 
The function $\phi$ corresponds to $\phi_0$ in the terminology of \cite{wei}.

\section{Proof of Theorem \ref{th2}}
Let $(M,\omega)$ be a $4$-dimensional compact and connected symplectic manifold.
Suppose that $J_0$ is an $\omega$-compatible almost-complex structure which induces a HEAK metric
with zero or negative Hermitian scalar curvature i.e. the Hermitian Ricci form $\rho^\nabla$ of $(\omega,J_0)$ satisfies $\rho^\nabla=\frac{s^\nabla}{4}\omega$ with $s^\nabla\leqslant 0.$
Moreover, suppose $h^-_{J_0}=b^+-1$, where $h^-_{J_0}$ is the dimension of $g_0$-harmonic $J_0$-anti-invariant $2$-forms \cite{dra-li-zha}. 
Let $J_t$ be a smooth family of $\omega$-compatible almost-complex structures in $AK_\omega$ such that $h^-_{J_t}=h^-_{J_0}=b^+-1$ for a small $t$. Denote by $g_t(\cdot,\cdot)=\omega(\cdot,J_t\cdot)$ the induced metric.

We consider the following almost-K\"ahler deformations
\medskip
\begin{equation*}
\omega_{t,f}=\omega+d\mathbb{G}_tJ_tdf,
\end{equation*}
where $\mathbb{G}_t$ is the Green operator associated to the Laplacian operator $\Delta^{g_t}$ with respect to the metric $g_t$ and $f\in{C}_0^\infty(M,\mathbb{R})$ a smooth function with zero mean value.

By Proposition \ref{p3}, the assumption $h^-_{J_t}=b^+-1$ implies that $\omega_{t,f}$ is $J_t$-invariant. Then, we define the map:

\begin{equation*}
\begin{array}{lrcl} \Phi : &\mathbb{R}\times{C}_0^\infty(M,\mathbb{R})& \longrightarrow &
{C}_0^\infty(M,\mathbb{R}) \\
    & (t,f)& \longmapsto
    &\mathring{s}^{\nabla_{t,f}},
\end{array}
\end{equation*}
\medskip
where $\mathring{s}^{\nabla_{t,f}}$ is the zero integral part of the Hermitian scalar curvature ${s}^{\nabla_{t,f}}$ of the almost-K\"ahler metric $(\omega_{t,f},J_{t})$. We have $\Phi(t,f)=0$ if and only if $(\omega_{t,f},J_{t})$ is an almost-K\"ahler metric with constant Hermitian scalar curvature. In particular, $\Phi(0,0)=0.$

Let ${W}^{p,k}$ be the completion of ${C}^\infty_0(M,\mathbb{R})$ with respect to the Sobolev norm $\|\cdot\|_{p,k}$ involving derivatives up to order $k$. Denote by $\Phi^{(p,k)}: \mathbb{R}\times{W}^{p,k+2}\longrightarrow{W}^{p,k}$ the extension of $\Phi$ to the Sobolev completion of ${C}^\infty_0(M,\mathbb{R})$. The map $\Phi^{(p,k)}$ is well defined when $pk>2n$. The kernel of the Laplacian $\Delta^{g_t}$ are $g_t$-harmonic $p$-forms and thus the dimension of the kernel of $\Delta^{g_t}$ is independent of $t$. Hence, we deduce from \cite[Theorem 7.6] {kod} that $\mathbb{G}_t$ is a $C^1$ map.
Thus, the map $\Phi^{(p,k)}$ is clearly a $C^1$ map.

Using the formula \ref{eq23} and definition of the Hermitan scalar curvature, we have the following

\begin{proposition}\label{pro5}
Let $(M,\omega,J,g)$ be a $4$-dimensional compact almost-K\"ahler manifold. Denote by $\mathbb{G}$ the Green operator associated to the Laplacian $\Delta^g$. Suppose that $d\mathbb{G}d^cf$ is $J$-invariant for any function $f$.
Then, for any almost-K\"ahler variation $\dot{\omega}=d\mathbb{G}d^c\dot{f}$ of the symplectic form $\omega$ ($\dot{f}$ with zero integral), the variation of the volume form, of the Hermitian Ricci form and the Hermitan scalar curvature are given by
\begin{eqnarray}
\dot{\big{(}{\omega^2}{}\big{)}}&=&(\delta^gJ\mathbb{G}d^c\dot{f})\,\,{\omega^2}{}=-\dot{f}\omega^2, \label{eq6}\\
\dot{\rho^{\nabla}}&=&\frac{1}{2}dd^c\dot{f}, \label{eq7}\\
\dot{s^{\nabla}}&=&-\Delta^g\dot{f}-2g(\rho^{\nabla},d\mathbb{G}d^c\dot{f}).\label{eq8}
\end{eqnarray}
\end{proposition}
Remark that, in the K\"ahler case, if we substitute $\dot{f}$ by $\Delta^g{\dot{\phi}}$, then the above variations coincide with the variation $\dot{\omega}=dd^c\dot{\phi}$ of the K\"ahler form $\omega$ in the (fixed) K\"ahler class \cite{leb-sim-2}.
\begin{proof}
Let $\omega_t=\omega+td\mathbb{G}d^c\dot{f}$. Then, using the relation \ref{eq21}, $\dot{\big{(}{\omega^2}{}\big{)}}=\frac{d}{dt}(\omega_t)^2|_{t=0}=g(d\mathbb{G}d^c\dot{f},\omega)\,\,\omega^2=(\delta^gJ\mathbb{G}d^c\dot{f})\,\,{\omega^2}{}$.
Moreover, using the fact that $d\delta^c+\delta^cd=0$ and the $J$-invariance of $d\mathbb{G}d^c\dot{f}$, we have
\begin{eqnarray*}
d\delta^gJ\mathbb{G}d^c\dot{f}&=&-d\delta^c\mathbb{G}d^c\dot{f}\\
&=&\delta^cd\mathbb{G}d^c\dot{f}\\
&=&J\delta^gd\mathbb{G}d^c\dot{f}\\
&=&J\Delta^g\mathbb{G}d^c\dot{f}\\
&=&Jd^c\dot{f}=-d\dot{f}
\end{eqnarray*}

Since $\dot{f}$ has zero integral, we obtain the second equality in \ref{eq6}. The variation of the Hermitian Ricci form \ref{eq7} follows from \ref{eq23} while the expression \ref{eq8} is a consequence of \ref{eq19}.
\end{proof}

Since $(\omega,J_0)$ is HEAK and $g(d\mathbb{G}d^c\dot{f},\omega)=-\dot{f}$, it follows from Proposition \ref{pro5} that the partial derivative $\frac{\partial \Phi }{\partial f}|_{(0,0)}$ is given by
\begin{equation}\label{der1}
 \frac{\partial \Phi }{\partial f}|_{(0,0)}(\dot{f})=-\Delta^{g_0}\dot{f}+\frac{s^{\nabla}}{2}\dot{f},
\end{equation}
where $s^{\nabla}$ is the Hermitian scalar curvature $(\omega,J_0)$. Clearly, $ \frac{\partial \Phi }{\partial f}|_{(0,0)}$ is a self-adjoint elliptic linear operator. Furthermore, it is an isomorphism of ${C}^\infty_0(M,\mathbb{R}).$ Indeed, suppose that
$-\Delta^{g_0}\dot{f}+\frac{s^{\nabla}}{2}\dot{f}=0$ for a function $\dot{f}$ (with zero integral), then $\Delta^{g_0}\dot{f}=\frac{s^{\nabla}}{2}\dot{f}.$ By hypothesis, $s^{\nabla}\leqslant 0$. As $M$ is compact, $\dot{f}\equiv 0.$ 
The natural extension of $ \frac{\partial \Phi }{\partial f}|_{(0,0)}$ to ${W}^{p,k+2}$ is again an isomorphism from ${W}^{p,k+2}$ to ${W}^{p,k}$. It follows from the implicit function theorem for Banach spaces that there exists $\epsilon,\delta>0$
such that for $|t|<\epsilon$,  there exists $f_t$ satisfying $\|f_t\|_{p,k}<\delta$ such that $\;\Psi^{(p,k)}(t,f_t)=0$. Hence, for each $|t|<\epsilon$, $(\omega_{t,f_t},J_t)$ is an almost-K\"ahler metric with constant Hermitian scalar curvature
of regularity ${W}^{p,k+2}$. It follows from the bootstraping argument used in \cite{lej1} that $(\omega_{t,f_t},J_t)$ are a family of smooth almost-K\"ahler metrics with constant Hermitian scalar curvature. Theorem \ref{th2} follows from the Moser Lemma.

\begin{example}
Theorem \ref{th2} may be applied to the Kodaira--Thurston manifold given by $S^1\times(Nil^3/\Gamma)$ where 
\begin{equation*}
Nil^3=\left\{\left(\begin{array}{ccc}1 & x & z \\0 & 1 & y \\0 & 0 & 1\end{array}\right),\,\,\,x,y,z\in\mathbb{R}\right\},
\end{equation*}
and $\Gamma$ is the subgroup of $Nil^3$ consisting of elements with integral entries. The $1$-forms $dt,dx,dy,dz-xdy$ are invariant under the action of $\Gamma$ (here $t$ is the $S^1$ coordinate). 

By a result of Li and Tomassini \cite{li-tom}, any homogeneous almost-K\"ahler structure $(\omega,J)$ on $S^1\times(Nil^3/\Gamma)$
has $h^-_J=b^+-1=1.$ Namely, the following symplectic form
\begin{equation*}
\omega=dx\wedge dt+dy\wedge(dz-xdy)
\end{equation*}
and the non-integrable $\omega$-compatible almost-complex structure
\begin{equation*}
Jdx= dt,\,\,\,\,\,\,\,\,Jdy=(dz-xdy).
\end{equation*}
verifies $h^-_J=b^+-1$. Moreover, the Hermitian Ricci form $\rho^\nabla$ of $(\omega,J)$ is zero \cite{dra,li,vez}.
\end{example}

\begin{remark}
When the Hermitian scalar curvature is positive one can prove the following:
Suppose that $T$ is a maximal torus in $Ham(M,\omega)$. Let $J_t$ be any smooth family of $\omega$-compatible $T$-invariant almost-complex structure such that
$J_0$ induces a HEAK metric with positive Hermitian scalar curvature and close enough in $C^{\infty}$-topology to an integrable
$T$-invariant $\omega$-compatible almost-complex structure $J$. Moreover, suppose that for a small $t$, $h^-_{J_t}=h^-_{J_0}=b^+-1.$
Then, there exists a smooth family of $\omega$-compatible $T$-invariant almost-complex structures ${\tilde{J}}_t$ inducing almost-K\"ahler metrics with constant Hermitian scalar curvature such that
$\tilde{J}_t$ is diffeomorphic to $J_t$ and $\tilde{J}_0=J_0$.

Observe that our hypothesis here implies that $(M,J)$ is a Fano complex surface. Since $b^+=1$, the condition $h^-_{J_t}=b^+-1=0$ is automatically satisfied for any family $J_t$. By Tian result \cite{tia}, there exists a K\"ahler-Einstein metric
except in the first Hirzebruch surface and its blown up at one point (actually these two surfaces are toric). In the latter two cases, the Futaki invariant of the anti-canonical class being non-zero implies that there is no toric HEAK metric \cite{lej}.

\end{remark}

\subsection*{Acknowledgements} The author would like to thank Prof. T.-J. Li for many helpful conversations and advices; Prof. V. Apostolov for his hospitality during a week at UQ\`AM and for his encouragement and suggestions. 

\bibliographystyle{cdraifplain}
\bibliography{xampl}

\end{document}